\documentclass[12pt]{amsart}
\usepackage[foot]{amsaddr}
\usepackage{amssymb, amsmath}
\usepackage{enumitem}
\usepackage{xspace} 
\usepackage{graphicx}
\usepackage{hyperref}

\newtheorem{thm}{Theorem}
\newtheorem{lem}[thm]{Lemma}

\newtheorem{cor}[thm]{Corollary}
\newtheorem{conj}[thm]{Conjecture}

\newtheorem{openp}{Question}

\newcommand{\wyt}{\textsc{wythoff}\xspace}
\newcommand{\nim}{\textsc{nim}\xspace}
\newcommand{\twyst}{\textsc{twyst-off}\xspace}
\newcommand{\nimmy}{\textsc{frozen twyst-off}\xspace}
\newcommand{\bigtwyst}{\textsc{heavy handed twyst-off}\xspace}
\newcommand{\untangle}{\textsc{untangle}\xspace}

\newcommand{\chomp}{\textsc{chomp}\xspace}

\newcommand{\be}{\begin{enumerate}}
\newcommand{\ee}{\end{enumerate}}

\begin{document}

\title{A New Twist on Wythoff's Game}
\author{Alex Meadows}
\address{Department of Mathematics and Computer Science\\ St. Mary's College of Maryland\\ St. Mary's City, MD 20686}
\email{ammeadows@smcm.edu}
\author{Brad Putman}
\email{bwputman@berkeley.edu}
\begin{abstract}
Wythoff's Game is a game for two players playing alternately on two stacks of tiles.  On her turn, a player can either remove a positive number of tiles from one stack, or remove an equal positive number of tiles from both stacks.  The last player to move legally wins the game.  We propose and study a new extension of this game to more than two stacks, which we call {Twyst-off}, inspired by the Reidemeister moves of knot theory.  From an ordered sequence of stacks of tiles, a player may either remove a positive number of tiles from one of the two end stacks, or remove the same positive number of tiles from two consecutive stacks.  Whenever an interior stack is reduced to 0, the two neighboring stacks are combined.  In this paper, we prove several results about those {Twyst-off} positions that can be won by the second player (these are called $P$ positions).   We prove an existence and uniqueness result that makes the visualization of data on three-stack $P$ positions possible.  This shows that many such positions are symmetric, like the easy general examples $(a,a,a)$ and $(a,a+1,a)$.  The main result establishes tight bounds on those three-stack $P$ positions that are not symmetric.  We go on to prove one general structural result for positions with an arbitrary number of stacks.  We also prove facts about the game when allowing stacks of infinite size, including classifying all positions $(\infty, \infty, \ldots, \infty)$ for up to six stacks.
\end{abstract}

\maketitle

\thispagestyle{empty}

\section{Introducing \twyst}

Wythoff's Game, which we abbreviate  \wyt, was invented by Wythoff in 1907 \cite{wythoff} as a modification of Bouton's classic game \nim \cite{bouton}.  Both articles include solutions of their proposed games. For a  nice problem-based introduction to both games, see the article on two games with matchsticks in  \cite{kvant}.   There is also a beautiful exposition of \wyt by Martin Gardner \cite{gardner}.  Both of these expository articles claim that the game might have been played in China long before 1907 under the name ``{Tsyan-shidzi}'', or ``{Choosing Stones}.''

\wyt is a game for two players playing alternately on two stacks of tiles.  On her turn, a player can either remove a positive  number of tiles from one stack, or remove an equal positive number of tiles from both stacks.  The last player to move legally wins the game.  The classical analysis of impartial combinatorial games (\cite{WW1}, \cite {guy}) partitions all of the game positions $(a,b)$ into \textit{$P$ positions} (from which the second or previous player can force a win) and \textit{$N$ positions} (from which the first or next player can force a win).   Every $N$ position can move to a $P$ position, but no $P$ position can move to another $P$ position.  
For example, $(0,0)$ is a $P$ position since the next player cannot move legally.  The positions $(n,0)$, $(0,n)$, and $(n,n)$ for $n>0$ are all $N$  since the next player can move to $(0,0)$.  The position $(1,2)$ is  $P$  since every \textit{option} is $N$.
A player can win the game by  moving from $N$ positions to $P$ positions, which guarantees that she wins.  Wythoff characterized the set of all $P$ positions in \wyt with the smaller stack listed first as the sequence $(a_n,b_n)=(\lfloor n\phi \rfloor, \lfloor n\phi^2 \rfloor)$, where $\phi=1.61803...$ is the golden ratio, the irrational positive root of the equation $\phi^{-1}+\phi^{-2}=1$.  Thus, on a grid of pairs $(a,b)$, the $P$ positions lie approximately on two lines of slope $\phi$ and $\phi^{-1}$.  
The sequence $(a_n,b_n)$ is called the sequence of \textit{Wythoff pairs}, and the numbers $a_n$ and $b_n$ are called \textit{lower and upper Wythoff numbers}, respectively.  For a gentle introduction to the interesting numerical properties of $(a_n,b_n)$, see \cite{mathellaneous} and references therein.

Many authors have studied variants of \wyt, including games produced by adding or subtracting legal moves, or by extending to more than two stacks of tiles.  In particular, Fraenkel introduced a compelling $n$-stack generalization in the same way that \wyt generalizes the classical game \nim \cite{MR2056935}.  His game is further analyzed in \cite{sun}, \cite{zeilberger}.

 In this paper, we introduce a new generalization to $n$ stacks called \twyst, inspired by the Reidemeister moves of knot theory and the similar game \untangle studied in \cite{ganzell1}, \cite{ganzell2}.  The idea derives from a new characterization of  \wyt.  Rather than playing on two stacks of tiles, we play \twyst on an unknotted but twisted loop of rope as in Figure~\ref{fig:l4r4}.  In Figure~\ref{fig:l4r4}, there are four \textit{left twists} in the rope followed by four \textit{right twists}.  On her turn, a player may either grab the left end of the rope and untwist some of the left twists, or grab the right end and untwist some of the right twists, or grab the piece between the last left twist and the first right twist, and unravel pairs of twists from the center.  The first  type of move corresponds to a sequence of type I Reidemeister moves applied to the same arc in the diagram.  The second type of move corresponds to a sequence of type II Reidemeister moves applied to the same arc in the center.  Thus, a player can either remove some of the left twists, remove some of the right twists, or remove an equal number of both.  Again, the last player to move legally wins.  Any such position of \twyst with $a$ left twists and $b$ right twists is identical to the position $(a,b)$ in \wyt.  

\begin{figure} 
\includegraphics{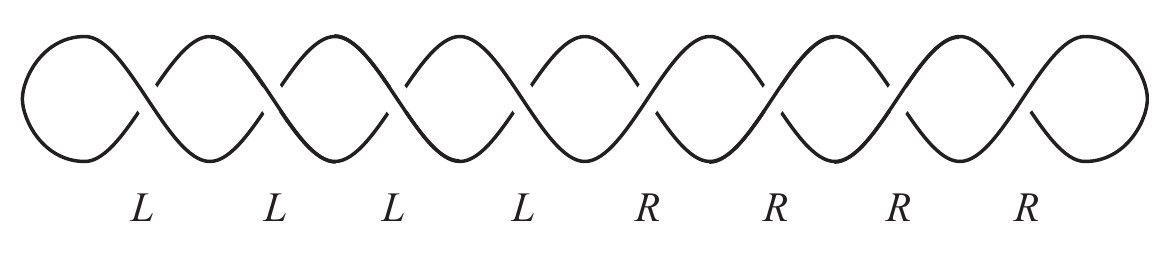}  
\caption{The \twyst position $(4,4)$ is equivalent to the position $(4,4)$ in \wyt}
 \label{fig:l4r4}
\end{figure}

Now we can generalize \wyt by allowing more complicated positions of \twyst, such as that in Figure~\ref{fig:l4r2l2}.  From this position, a player can either remove some of the left twists from the left side, or some of the left twists from the right side, or the same number of left and right twists from two consecutive sequences of twists.  We can play an equivalent game on stacks of tiles with sizes $(4,2,2)$.  Notice that the next player cannot simply remove tiles from the middle stack.  Thus, \twyst is not a modification of \nim since not all \nim moves are included.   Also, if she removes two tiles from the left two stacks, the two remaining tiles in the left stack combine with the two in the right stack (they are both sequences of left twists) to make the position $(4)$.  Thus, unlike many \nim-like games, \twyst does not break into a disjunctive sum of games after certain moves.  The complete  list of options to which the player can move from $(4,2,2)$ is $\{(3,2,2), (2,2,2), (1,2,2), (2,2), (4,2,1), (4,2), (4,1,1), (4), (3,1,2)\}$.  

\begin{figure}
\includegraphics{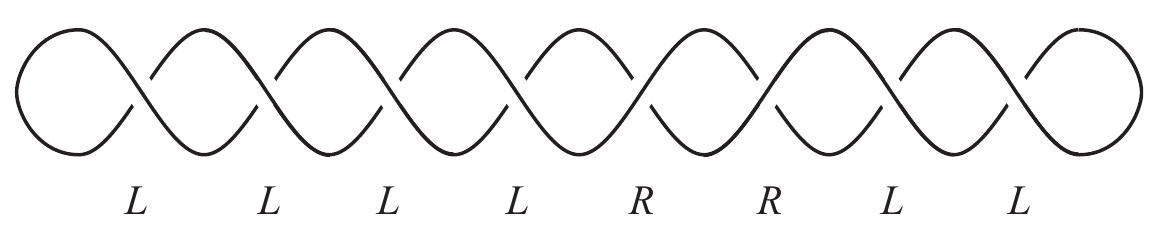}
\caption{The \twyst position $(4,2,2)$}
\label{fig:l4r2l2}
\end{figure}

For the remainder of the paper, we will not consider knot diagrams when analyzing the game.  In general, the game \twyst is played on an ordered sequence of stacks of tiles.  A player's turn consists of either removing any number of tiles from one of the end stacks, or removing an equal number of tiles  from two consecutive stacks.  If an interior stack is reduced to zero tiles, the two stacks adjacent are combined into a single stack.  That is, a position like $(a,b, 0, c,d)$ is just $(a,b+c,d)$. We call this equivalence  \textit{contraction}.  In normal play, the last player to move wins.  To prove results, we will regularly use either the well-ordering principle or induction, as explained in the game theory context in \cite{lessons}.

\section{Three Stack \twyst and a Wealth of Data}

Although \twyst can be played on any ordered sequence of stacks, we will find much to interest us just by looking at three stack positions.  We will denote the general three-stack position $(a,b,c)$ and try to characterize those that are $P$ positions.  Note that in general we can prove that a position is $P$ by showing that all of its options are $N$, and we can prove that a position is $N$ by showing that one of its options is $P$.

For example, consider the positions with three equal stacks. It is not hard to see that these  will always be $P$. Suppose there existed a position of this form that was $N$.  Let's consider the smallest such position and call it $(a, a, a)$. We know that the only options (up to symmetry and if they exist) from this position are of the form $(a- k, a, a)$ or $(a - k, a - k, a)$. Each has the option $(a-k, a-k, a-k)$, which has to be $P$ by assumption. Consequently, all of the options of $(a, a, a)$ are N, which contradicts our assumption that $(a, a, a)$ is  $N$. Thus, we have:

\begin{thm}  \label{equaltriples}
 For all $a \geq 0$, $(a, a, a)$ is ${P}$.
\end{thm}

As we will see shortly, it is very common for $3$-stack $P$ positions to be \textit{palindromes}, i.e.\ of the form $(a,b,a)$.  For example, we can use a similar argument to that above to show that $(1,b,1)$ is $P$ for all $b\ge 4$.  Note that $(1,2,1)$ and $(1,3,1)$ are not $P$ since both have $(1,2)$ as an option.  In fact, so does $(2,3,2)$.  Also, we can show by exhaustion that $(1,3,2)$ is a $P$ position and another $P$ option for $(2,3,2)$.  The next result shows that $(0,1,0)$, $(1,2,1)$, and $(2,3,2)$ are the only palindromes $(a,a+1,a)$ that are $N$.

\begin{thm}
 For all $a \geq 3$, $(a, a+1, a)$ is ${P}$.
\end{thm}

\begin{proof}
The options of position $(a,a+1,a)$ up to symmetry are $(a-k,a+1,a)$ and $(a-k, a-k+1,a)$ for $0<k\le a$, which we must prove are $N$.  Each of these has the option $(a-k, a-k+1,a-k)$, which is a $P$ position by induction unless $a-k=0$, $1$, or $2$.  In case $a-k=0$, the option $(0,a+1,a)$ has $P$ option $(2,1)$, and $(0,1,a)$ has the $P$ option $(1,2)$.  If $a-k=1$, then $(1,a+1,a)$ has the $P$ option $(1, a+1,1)$ (note that $a+1\ge 4$), while $(1,2,a)$ has the $P$ option $(1,2)$.  Finally, if $a-k=2$, then $(2,a+1,a)$ has the $P$ option $(2,1)$, while $(2,3,a)$ has the $P$ option $(2,3,1)$.  Thus, in all cases the options of $(a,a+1,a)$ are $N$.
\end{proof}

The next result is our main existence and uniqueness theorem that allows us to represent $P$ positions visually.

\begin{thm} \label{cexists}
If $a > 0$ and $b > 0$, then there exists a unique $c \geq 0$, such that $(a, b, c)$ is  ${P}$. Furthermore, $c \leq a + b + \min \{a, b \} + 1$.
\end{thm}

\begin{proof}
Such a $c$ must be unique, for otherwise there would be two $P$ positions $(a,b,c)$ and $(a,b,c+k)$, the first of which is an option of the second. To see existence, we first let $d = a + b + \min \{a, b \}$ for clarity. Suppose that $(a, b, 0), (a, b, 1), ... \, , (a, b, d)$ are all $N$ positions. By definition, each must have a move that results in a $P$ position. We argue that each of these moves must be distinct. Without loss of generality, suppose $(a, b, c)$ and $(a, b, c')$ can both move to $P$ positions by removing $k$ tiles from the left and middle stacks. Thus, $(a - k, b - k, c)$ and $(a - k, b - k, c')$ are both $P$ positions. By the uniqueness property, we conclude that $c = c'$. So, the moves that take $(a, b, 1), ... \, , (a, b, d)$ to their respective $P$ options must be distinct. Finally, consider $(a, b, d + 1)$. We know that removing from only the right-most stack will result in an $N$ position. A quick count of the remaining options yields $d$ possible moves which could result in a $P$ position. But notice that this set of $d$ possible moves is the same set of $d$ possible moves by which $(a, b, 1), ... \, , (a, b, d)$ could be reduced to a $P$ position. We've already established that if $(a, b, d+1)$ has a move to a $P$ option, then it must be distinct from these. Thus, by the pigeonhole principle, we know that every option of $(a, b, d+1)$ is an $N$ position, so it must itself be $P$.
\end{proof}

Using the same ideas, we can extend Theorem~\ref{cexists} to longer positions $(a_1, \ldots, a_n)$ where the unique $a_n$ is bounded in terms of the other stack sizes.  We can also  prove the following.  Recall that the Grundy value of a game $G$ is the unique number $k$ such that the game $G+\textrm{\textsc{nim}}(k)$ is $P$.  

\begin{thm} \label{grundy}
Given positive $a_1, \ldots, a_{n-1}$ and $g\ge 0$, there is a unique $a_n\ge 0$, bounded in terms of the other numbers, such that the \twyst position $(a_1, \ldots, a_{n-1}, a_n)$ has Grundy value $g$.
\end{thm}

\begin{proof}[Proof idea:]
We apply the same analysis to the game $(a_1, \ldots, a_{n})+(g)$, where $(g)$ is just \nim played on one stack with $g$ tiles.  A unique $a_n$ will make this a $P$ position.
\end{proof}

Theorem~\ref{cexists} allows us to define the function $f: \mathbb{Z}^{> 0}\times \mathbb{Z}^{\geq 0} \rightarrow \mathbb{Z}^{\geq 0}$  by $f(a, b) = c$, where $(a, b, c)$ is ${P}$. The table of values of $f$ (shown on the left of Figure~\ref{data1}) makes it much easier to visualize possible patterns in the $P$ Positions (much like the chessboard diagram in standard \wyt).  We find some interesting patterns in the array in Figure~\ref{data1}.  Most notable is that it appears to be common for $f(a,b)=a$, i.e.\ many $P$ positions are \textit{palindromes} $(a,b,a)$. Of course, this cannot happen when $(a,b)$  is a $P$ position for \wyt, in which case $f(a,b)=0$.  Also, it cannot happen when $(a,b-a)$ is a $P$ position for \wyt, in which case $(a,b,a)$ has a $P$ option.  We color the data to highlight these patterns.  The white squares represent positions that are palindromes and the gray represent those that are not.  The pairs where $f(a,b)=0$ (the $P$ positions for standard \wyt) are colored green, and the positions where $(a,b-a)$ is $P$ for \wyt  are colored orange.  Figures~\ref{data2} and  \ref{data3} show larger collections of data.

\begin{figure} 
\includegraphics[scale=.5]{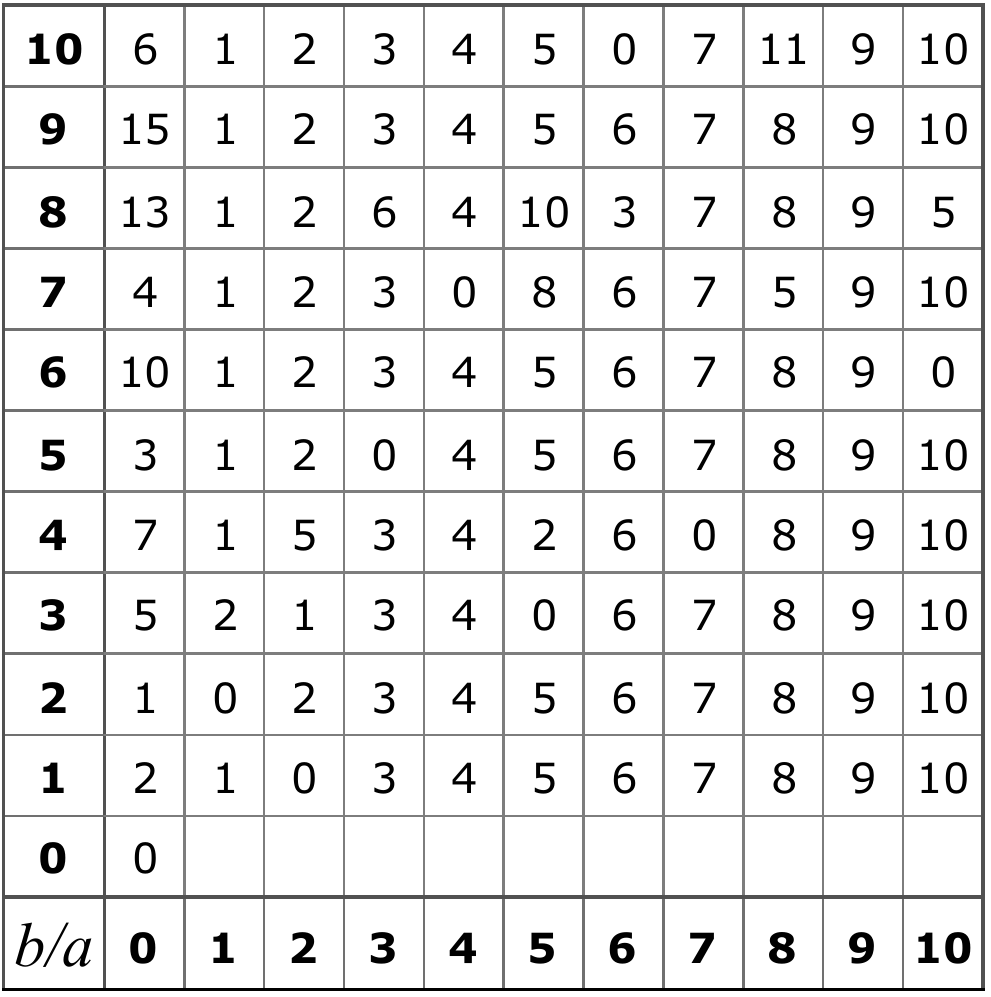}  \hspace{.2in}
\includegraphics[scale=.5]{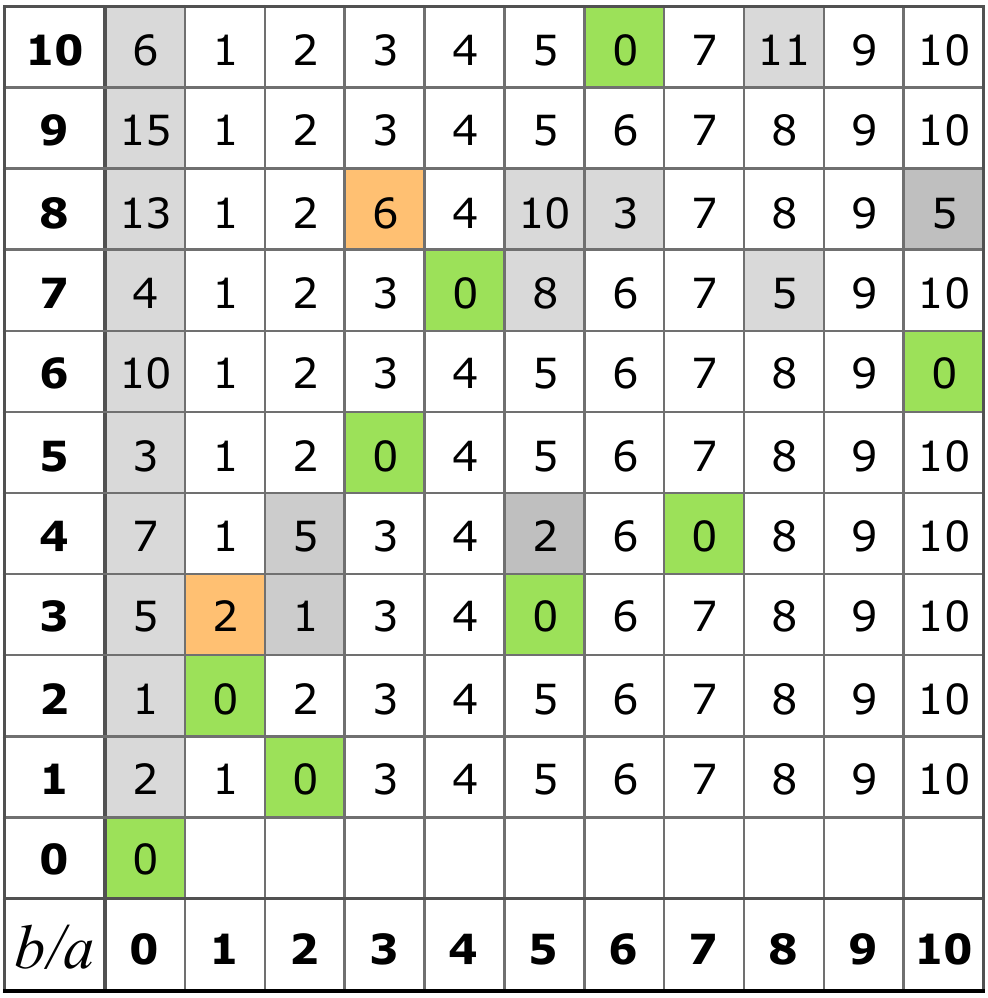}  
\caption{A table of values $c$ such that $(a,b,c)$ is $P$, guaranteed by Theorem~\ref{cexists}}
 \label{data1}
\end{figure}

\begin{figure} 
\includegraphics[scale=.35]{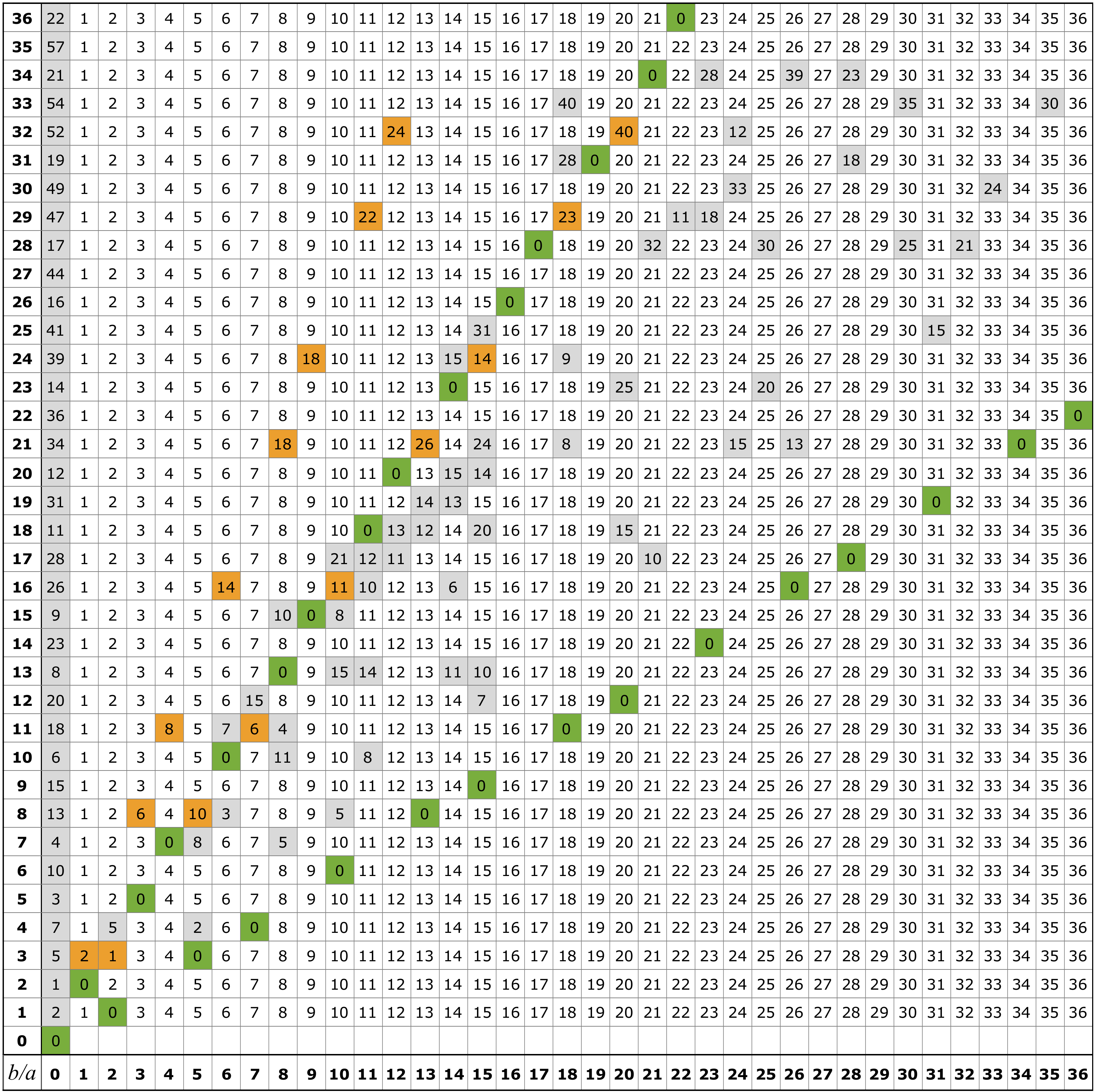}  
\caption{A larger table of $P$ positions}
 \label{data2}
\end{figure}

\begin{figure} 
\includegraphics[scale=.05]{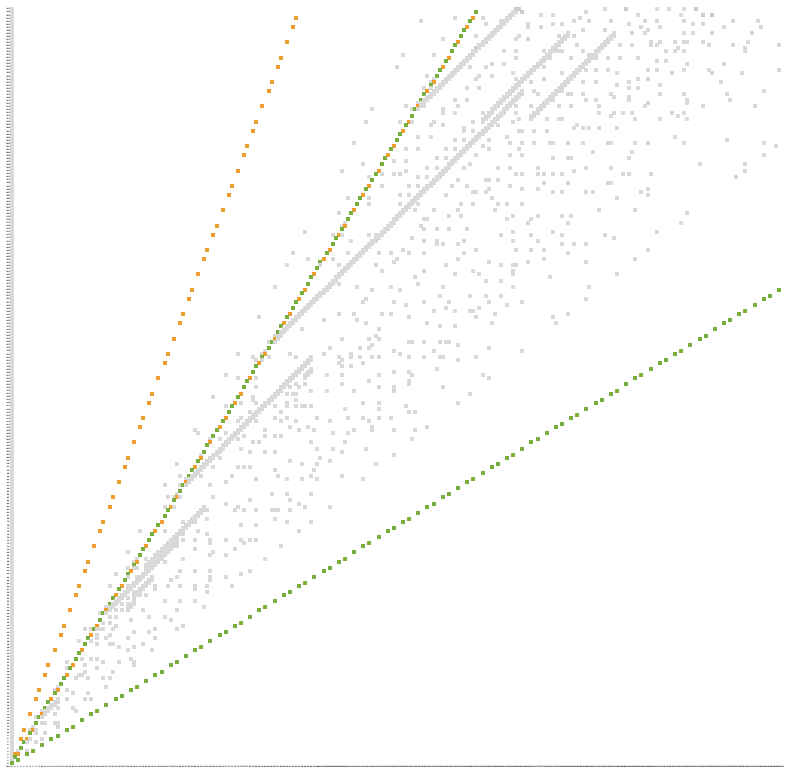}  
\caption{An even larger table}
 \label{data3}
\end{figure}

Notice how the symmetry of the game rules appears in the graphs.  That is, when $(a,b,c)$ is a $P$ position ($f(a,b)=c$), so is $(c,b,a)$ ($f(c,b)=a$).  So, if $a\ne c$ then row $b$ of the graph will have $c$ in column $a$ and $a$ in column $c$.  Also, by Theorem~\ref{cexists}, fixing $b$, for each $c$ there is exactly one $a$ such that $f(a,b)=c$.  So, each row of the graph is a permutation of the nonnegative integers in order  composed entirely of disjoint involutions.  For example, in row $7$, $0$ is exchanged with $4$ and $5$ is exchanged with $8$.  In particular, 
recall from the theory of standard \wyt that each row will contain exactly one pair $(a,b)$ for which $f(a,b)=0$.  This $a$  will then be $f(0,b)$, in the grey column on the left side of the data.  Thus, the column on the left is a permutation of the nonnegative integers given by the set of all Wythoff pairs as involutions, or integer sequence  A002251\cite{oeisA002251}.  

Notice that below a line of approximate slope $\phi^{-1}$ (populated by green squares) and above a line of approximate slope $\phi^2$ (populated by orange squares), all positions appear to be palindromes.  These facts are the main results of this paper.  As a consequence, each row of the graph is a permutation of the nonnegative integers composed of a finite number of disjoint involutions.  Also, for each column $a$ of the graph, we have $f(a,b)=a$ for all but finitely many $b$.  There appear to be further patterns in the data that repeat at higher scales similar to those studied in \cite{friedman}.  These patterns are a subject of ongoing research.  One pattern that is not quite obvious from the graph is given in the lemma below.

\begin{lem}
If $(a,b,c)$ is a non-palindrome $P$ position with $a<c$, then $a<b$.
\end{lem}

\begin{proof}
Suppose $(a,b,c)$ is $P$ with $b\le a<c$.  By Theorem~\ref{equaltriples}, we must have $b<a<c$.  Further suppose that this is such a position with smallest value of $b$.  Then the position $(a,b,a)$ must be $N$ and therefore have a $P$ option.  The $P$ option is either $(a,b,t)$ for $t<a$, which is a contradiction since we assumed that $(a,b,c)$ is $P$, or it is $(a-k,b-k,a)$ for some $k\le b<a$.  In case $k=b$, we get the single stack position $(2a-k)$, which cannot be $P$.  In case $k<b$, we get a contradiction of the minimality of the example $(a,b,c)$.  
\end{proof}

The standard $P$ positions of \wyt give us an infinite collection of palindromes that are not $P$ positions for \twyst.  That is, if $(a_n,b_n)$ is $P$, then the positions $(a_n,b_n,a_n)$, $(b_n,a_n,b_n)$,  $(a_n, a_n+b_n, a_n)$, and $(b_n, a_n+b_n,b_n)$ are all $N$.  The corresponding non-palindrome $P$ positions $(a_n,b_n,c)$ and  $(b_n,a_n,c)$ are colored green in the figures above, and $(a_n, a_n+b_n, c)$ and $(b_n,a_n+b_n, c)$ are colored yellow.  From Wythoff's original description of $(a_n,b_n)$ as $(\lfloor n\phi\rfloor, \lfloor n\phi^2\rfloor)$, we have $b_n=\lceil \phi a_n\rceil$, $a_n=\lfloor b_n/\phi \rfloor$, $a_n+b_n = \lceil \phi^2 a_n\rceil$, and $a_n+b_n=\lceil \phi b_n\rceil$.  These give us approximate lines of non-palindrome $P$ positions in the figures of slopes $\phi$, $\phi^2$ and $\phi^{-1}$.  Furthermore, since every positive integer is a member of a Wythoff pair, the line of slope $\phi$  contains a point $(x,y)$ for every $x$.  The other two lines are extremal, as shown after the next lemma, which summarizes the discussion above.

\begin{lem} \label{tripleexamples}
Let $a$ be a positive integer. Then 
\begin{enumerate} \renewcommand{\labelenumi}{\textup{(\arabic{enumi})}}
\item $(a, \lceil \phi a \rceil, a)$ is ${N}$. 
\item If $a=a_n$ is a lower Wythoff number, then $(a,\lceil \phi^{2}a \rceil, a)$ is ${N}$.
\item If $a=b_n$ is an upper Wythoff number, then $(a,\lfloor \frac{a}{\phi	} \rfloor, a)$ is ${N}$.
\end{enumerate}
\end{lem}

\begin{lem}  \label{lowerbd}
If $b > \lceil \phi^{2}a \rceil$, then $(a, b, a)$ is ${P}$.
\end{lem}

\begin{proof}
Suppose the hypothesis holds. We examine the possible options (up to symmetry) of $(a, b, a)$. Removing the entire first column results in the position $(0, b, a)$. Since $b > \lceil \phi a \rceil > a > \lfloor \frac{a}{\phi^{2}} \rfloor$, this cannot be a standard \wyt $P$ Position. Removing $j$ from the first column, for $1 \leq j < a$, results in $(a - j, b, a)$. This position must be $N$ because it has the option $(a - j, b, a-j)$, which is $P$ by induction. Removing $a$ from the first and second columns results in the position $(0, b- a, a)$. We have already shown that $\lceil \phi^{2} a \rceil = \lceil \phi a \rceil + a$. So $b - a > \lceil \phi a \rceil + a - a = \lceil \phi a \rceil$, therefore this cannot be a standard \wyt $P$ Position. Finally, removing $k$ from the first and second columns, for $1 \leq k < a$, results in the position $(a - k, b - k, a)$. This position must be $N$ because it has the option $(a - k, b -k, a - k)$, which is $P$ by induction. 
\end{proof}

\begin{lem}  \label{upperbd}
If $0 < b < \lceil \frac{a}{\phi} \rceil$, then $(a, b, a)$ is ${P}$ \label{lowerbound}
\end{lem}

This lemma is more difficult and requires preliminary results. In the proof of the existence and uniqueness theorem, we established the fact that if $(a, b, c)$ is a $P$ position, then $c$ is bounded above by the inequality $c \leq a + b + \min \{a, b \}+1$. In order to prove this lemma, we will need a stronger bound for $c$. We establish one in the next section, namely

\begin{lem} \label{twystinequality}
If $(a,b,c)$ is a $P$ position for \twyst with $a>0$, then $c<a+b$.
\end{lem}

\begin{proof}[Proof of Lemma~\ref{lowerbound}]
We proceed by induction.  Consider the position $(a,b,a)$ where $a > \lceil \phi b \rceil$.   Then for any $k<\min\{a,b\}$, we have $a-k > \lceil \phi (b-k) \rceil$.  Thus, by induction we may assume that any legal position $(a-k, b-k,a-k)$ is $P$.  Now we show that the options of $(a,b,a)$ are $N$ positions.  Any option $(a-k,b-k,a)$ may be moved to the $P$ position $(a-k,b-k,a-k)$ unless either $k=b$ or $k=a$.  If $k=b$, the position $(a-k,0,a)=(2a-k)$ is a single stack $N$ position.  If $k=a$, the position is $(b-a,a)$ which cannot be a \wyt $P$ position since $a>\lceil \phi(b-a)\rceil$.  The remaining options of $(a,b,a)$ are positions $(a-k, b, a)$ for $1\le k\le a$.  If $k=a$, the position is the two-stack $(b,a)$, which cannot be a \wyt $P$ position since $a > \lceil \phi b \rceil$.  If $1\le k<b$, then $(a-k,b,a)$ can be moved to the $P$ position $(a-k,b-k,a-k)$.  Finally, if $b\le k$, then $a\ge a-k+b$, and so by Lemma~\ref{twystinequality}, $(a-k,b,a)$ must be an $N$ position.
\end{proof}

\begin{thm}
Let $(a, b, c)$ be a non-palindrome $P$ Position in \twyst. Then either $a = 0$, or $b$ is bounded by the inequalities $\lceil \frac{a}{\phi} \rceil \leq b \leq \lceil \phi^{2}a \rceil$.  Both inequalities are sharp for infinitely many positions. 
\end{thm}

\begin{proof}
Apply Lemmas~\ref{tripleexamples}, \ref{lowerbd}, and \ref{upperbd}.
\end{proof}

The proof of Lemma~\ref{twystinequality} will rely on showing that the $P$ positions do not change much when we change the rules of the game slightly for three stacks.  In fact, computational data supports a conjecture  that \twyst has the same $P$ positions as a more restrictive game on three stacks that we call \bigtwyst.  In \bigtwyst, you can only remove an equal number of tiles from a pair consecutive stacks if you are removing from the larger two pairs that share the same stack.  The other moves remain the same.  So for example you can move from $(2,3,5)$ to $(2,2,4)$, but not to $(1,2,5)$.

\begin{conj}
 \bigtwyst has the same $P$ positions as \twyst on three stacks.
\end{conj}

\section{Proving Palindrome Bounds}

We introduce the game \nimmy  to prove Lemma~\ref{twystinequality}.  \nimmy is identical to \twyst except that contraction does not happen with interior zero-tile stacks -- empty stacks are only removed from the ends of the sequence.  So, for example, in \twyst, the position $(3,0,3)=(6)$ is an $N$ position, while in \nimmy, $(3,0,3)$ is equivalent to $(3,3)$ in \textsc{nim} and so is a $P$ position.

\begin{lem} \label{uniquec}
For every pair $(a,b)$, there is a unique $c$ such that $(a,b,c)$ is a $P$ position for \nimmy.
\end{lem}

This lemma has the same proof as that of Theorem~\ref{cexists}.

\begin{lem} \label{nimmyinequality}
In \nimmy, if $(a,b,c)$ is $P$ and $b>0$, then $c<a+b$.
\end{lem}

\begin{proof}
Suppose there exist $P$ positions $(a,b,c)$ for \nimmy with $c\ge a+b>0$.  Let $(a_0,b_0,c_0)$ be a smallest such $P$ position, i.e. one with the smallest total sum $a+b+c$.  Then $(a_0,b_0,c_0-1)$ is an option and must be an $N$ position.  Thus, it has a $P$ option $(a',b',c')$.  But the $P$ option $(a',b',c')$ can't be $(a_0,b_0,k)$ for $k<c_0-1$ since that would be an option of $(a_0,b_0,c_0)$.  Thus, the move to $(a',b',c')$ must reduce the sum of the first two stacks at least as much as it reduces the third stack.  Thus, $c_0-1-c' \le a_0+b_0-a'-b'$.  Also, by induction, $c'<a'+b'$.  Adding the previous two inequalities, we get $c_0-1<a_0+b_0$, which together with $c_0\ge a_0+b_0$ gives $c_0=a_0+b_0$.  However, now if we remove $b_0$ from the two right stacks of $(a_0,b_0,c_0)$, we get $(a_0,0,a_0)$, which is a $P$ position for \nimmy.  Thus, we have a contradiction and all  $P$ positions with $b>0$ must satisfy $c<a+b$.
\end{proof}

\begin{thm} \label{sameps}
Except for the positions $(a,0,a)$, \twyst and \nimmy have exactly the same three-stack  $P$ and $N$ positions.
\end{thm}

\begin{proof}
Suppose that $(a,b,c)$ is a $P$ position for \nimmy and an $N$ position for \twyst.  Then it has an option that is a $P$ position for \twyst.  By induction, that option is also a $P$ position for \nimmy, which is a contradiction.

Suppose that $(a,b,c)$ is a $P$ position for \twyst and an $N$ position for \nimmy.  Then it has an option that is a $P$ position for \nimmy.  By induction, that option is either a $P$ position for \twyst (which would be a contradiction), or it has the form $(a,0,a)$, in which case $c=a+b$.  By Lemma~\ref{uniquec}, there must exist $c'$ such that $(a,b,c')$ is a $P$ position for \nimmy.  If $c'<c$, then again by induction, $(a,b,c')$ is a $P$ position for \twyst and we have a contradiction.  If $c'>c$, then \nimmy has a $P$ position $(a,b,c')$ with $c'>a+b$, contradicting Lemma~\ref{nimmyinequality}.  In all cases we get a contradiction, thus proving the theorem.
\end{proof}

Lemma~\ref{twystinequality} is now a corollary of Theorem~\ref{sameps} and Lemma~\ref{nimmyinequality}.

\section{More Stacks, More Fun}

As the three-stack version of our game is significantly more complicated than \wyt, it is not surprising that introducing more stacks makes the game more difficult. For example, consider a move on a three-stack position that results in an empty middle stack. The ensuing contraction  results in a one-stack position. Since we began with three positive stacks, and we can only remove from two stacks at a time, we know that the resulting one-stack position will always be positive, and thus an $N$ position. Consequently, when we want to know if a particular three-stack position is $P$, we don't need to worry about a move that empties the middle stack. When we venture beyond three stacks, we lose this advantage. 

Exhaustion shows that the smallest four-stack $P$ positions are $(1, 2, 2, 1)$, $(2, 1, 1, 2)$, $(1, 1, 2, 2)$, $(1, 2, 1, 2)$, and $(1,1,3,1)$.  These begin an inductive proof of 

\begin{thm} \label{4symmetric}
For $a,b>0$, the position $(a, b, b, a)$ is $P$ if and only if  $a = 1, b = 2$ or $a > 1, b = 1$.
\end{thm}

\begin{proof}
Suppose $(a,b,b,a)$ does not have the proposed form.  Then either $a=1$ and $b\ne 2$ or $a>1$ and $b>1$.  In the first case, we note that $(1,1,1,1)$ is $N$ since it has the $P$ option $(1,1,1)$, and if $b>2$, $(1,b,b,1)$ has the option $(1,2,2,1)$.  In the second case, the position has option $(a,1,1,a)$ which is $P$ by induction.

For $a>2$, the position $(a,1,1,a)$ has  up to symmetry the $N$ options $(a,a)$, $(1,1,1,a)$, and $(1,1,a)$.  These last two are $N$ because they  have the $P$ option $(1,1,1)$.  The remaining options have the form $(a',1,1,a)$ for $2\ge a'<a$.  These can be moved to $(a',1,1,a')$, which is $P$ by induction.
\end{proof}

Remaining four-stack $P$ positions are  messy but appear to follow some patterns in the long run.  For example, we can show that $(1,a,b,1)$ is $P$ if and only if $a=b=2$ or $a=3$ and $b>5$.  

\begin{conj}
For every $a>0$, there exists $c_0>b_0>a$ such that $(a,b_0,c,a)$ is $P$ if and only if $c\ge c_0$
\end{conj}

Next we consider \twyst on arbitrarily many stacks, and prove one general structural result.

Let $S$  be the set of \twyst positions $(s_1, \ldots, s_n)$ for which $1 \leq s_{i} \leq 2$ and any interior sequence of 1's has even length.

\begin{lem}
Every option of a position in $S$ is in $S$.
\end{lem}

We omit the proof, which just involves checking several types of moves. Let $S_k$ be the set of positions $(s_1, \ldots, s_n)$ in $S$ such that $\sum s_i = k$.  

\begin{lem} \label{sbase}
All positions in $S_0$ are $P$ and all positions in $S_1$ and $S_2$ are $N$.
\end{lem}

\begin{proof}
Exhaustion.
\end{proof}

\begin{thm} \label{general} If $k \equiv 0 \ (\text{mod } 3)$, then all positions in $S_k$ are $P$, and if $k \not \equiv 0 \ (\text{mod } 3)$, then all positions in $S_k$ are $N$.
\end{thm}

\begin{proof}
Induction using Lemma~\ref{sbase} for base cases.
\end{proof}

Notice that, a fortiori, we get some interesting classifications about positions of arbitrary length. For example, an arbitrary concatenation of certain elementary $P$ positions like $(2,2,2)$, $(2,1,1,2)$, and $(2,1,1,1,1,1,1,1,1,2)$ is $P$.

\section{Twisting to Infinity}

We now consider \twyst with infinite stack sizes.  Although the rules do not extend to 
 transfinite ordinals as naturally as they do for \chomp (see \cite{huddle}), there is a relatively natural extension that allows the game to end in a finite number of moves, with some interesting consequences. Here, we only consider including the first infinite ordinal, which we simply call $\infty$.  Reducing a stack of size $\infty$ means reducing it to any finite size.  Thus, the rules for \twyst extend to positions that may have stack size $\infty$ as follows.  On a single move, 
 
 \begin{itemize}
 \item an end stack of size $\infty$ may be reduced to any finite size;
 \item a pair of neighboring stacks $(\infty,\infty)$ may be reduced to equal finite sizes $(a,a)$;
 \item a pair of neighboring stacks $(\infty, b)$ where one size is finite may be reduced to a pair with a smaller finite stack $(\infty, b')$ with $0\le b'<b$; and
 \item as before, neighboring stacks combine due to the reduction of an interior stack to $0$.  If one of the neighboring stacks has infinite size, the resulting stack is $\infty$.  
 \end{itemize}
 
Here are some examples.  The only options (up to symmetry) from the position $(\infty, \infty, \infty)$ are $(a,\infty,\infty)$ and $(a, a, \infty)$ for some finite $a$.  These both have the $P$ option $(a,a,a)$, and so $(\infty,\infty,\infty)$ is $P$.  The position $(1,2,\infty,3,4,5)$ allows the player to make a standard \twyst move on either the $(1,2)$ part or the $(3,4,5)$ part, leaving the other part unchanged.  However, any move that reduces the $2$ or the $3$ to $0$ also removes the neighboring stack; for example $(1,0,\infty,3,4,5)$ becomes $(\infty,3,4,5)$ and $(1,2,\infty,0,4,5)$ becomes $(1,2,\infty,5)$.
 
 \begin{lem} \label{lastinfinity} The  position $(a_{1}, ... \, a_{n-1}, \infty)$ is  $N$.
 \end{lem}

\begin{proof} By Theorem \ref{cexists}, there exists a unique $a_{n}$ such that $(a_{1}, ... \, a_{n-1}, a_{n}) $ is $P$. This position is an option of $(a_{1}, ... \, a_{n-1}, \infty)$, which in turn must be $N$.
 \end{proof}
 
We will use the notation $(\alpha)$ and $(\beta)$ to represent sequences of stacks of finite size.  Lemma~\ref{lastinfinity} implies that, when playing the game $(\alpha,\infty, \beta)$, the goal is to not be the player who eliminates one of the two sides $\alpha$ or $\beta$.  For example, from the position $(3,2,\infty,1)$, removing a stack of size $2$ or $1$ leads to an $N$ position.  It turns out that the only $P$ option is $(2,2,\infty, 1)$.  We can thus consider the game $(\alpha,\infty,\beta)$ as a disjunctive sum of the games $(\alpha)$, and $(\beta)$ with modified rules: removing the right-hand stack of $(\alpha)$ also removes the neighboring stack, and the same with the left-hand stack of $(\beta)$.  However, the goal is to force the opponent to finish one of the two games in the sum.  Thus, the game $(\alpha,\infty,\beta)$ is $N$ or $P$ as is the \textit{misere diminished disjunctive sum} of the modified games $(\alpha)$ and $(\beta)$.  By symmetry, it is clear that $(\alpha,\infty, \alpha)$ is $P$.  Other $P$ positions $(\alpha,\infty,\beta)$ occur when the modified games $(\alpha)$ and $(\beta)$ have the same \textit{foreclosed Grundy number}.  We use this fact in the proof of Theorem~\ref{6infinities}.  See \cite{onag}, chapter 14 for the theory of diminished disjunctive sums.
 
 The game $(\infty, \alpha, \infty)$ is also interesting.  A move on either infinite stack results in an $N$ position by Lemma~\ref{lastinfinity}.  Thus, again, the goal is to force the other player to move last on $\alpha$, with the modified rules that removing an end stack of $\alpha$ also removes the adjacent stack.  That is, we are playing misere \twyst on $\alpha$ with these modified rules.  Somewhat surprisingly, we can characterize all $P$ positions $(\infty, \alpha, \infty)$ when $\alpha$ has length 3 or less, as completed by the following theorem.
 
 \begin{thm}  \label{outerinfinities} Suppose $a,b,c>0$.  Then  
\begin{enumerate} \renewcommand{\labelenumi}{\textup{\arabic{enumi}.}}
 \item $(\infty,a,\infty)$ is $P$ if and only if $a=1$.
 \item $(\infty,a,b,\infty)$ is $P$ if and only if $(a,b) = (c+1,d+1)$ where $(c,d)$ is $P$ for $\wyt$.
 \item $(\infty,a,b,c,\infty)$ is $P$ if and only if $a=c>1$.
 \end{enumerate}
 \end{thm}
 
 \begin{proof}
 Part 1 is clear.  For part 2, note that $(\infty, 1, 1, \infty)$ is $P$, and reducing either stack of $(a,b)$ to zero results in $(\infty,\infty)$, which is $N$.  The result follows by induction.  For part 3, note that if either $a$ or $c$ is $1$, then the other one can be reduce to zero, leaving the $P$ position $(\infty, 1,\infty)$.  Then if $a>1$, some options of the position $(\infty,a,b,a,\infty)$  are  $N$ because $b$ has reduced to zero or one of the $a$'s has reduced to $1$ or $0$.  All other options are of the form $(\infty,a,b,c,\infty)$ where $a,c>1$ and $a\ne c$, which are $N$ by induction.
 \end{proof}
 
 There are a few easy consequences of this theorem for cases when $\alpha$ has length $4$.  Namely,
 
 \begin{cor}  Suppose $a,b>0$.  Then  
\begin{enumerate} \renewcommand{\labelenumi}{\textup{\arabic{enumi}.}}
  \item $(\infty,1,a,b,1,\infty)$ is $P$ if and only if $a=2$ and $b>2$, or vice versa.
 \item $(\infty,a,1,b,1,\infty)$ is $P$ if and only if $a=2$ and $b\ge 2$.
 \item $(\infty,a,1,1,b,\infty)$ is $P$ if and only if $a=b>1$.
  \end{enumerate}
 \end{cor}
 
 Using Theorem~\ref{equaltriples}, we found that $\infty^3= (\infty, \infty, \infty)$ is $P$, and thus $\infty^4$ and $\infty^5$ are $N$.  It may come as a slight surprise that $\infty^6$ is also $N$, because it has the following $P$ option.
 
 \begin{thm} \label{6infinities}
 The position $(\infty, \infty, \infty, \infty, \infty, \infty)$ is $N$.
 \end{thm}
 
 \begin{proof}
 We will show that the option $(\infty, \infty, 1, 1, \infty, \infty)$ is $P$.  
 We have already seen that the option (by removing either or both 1's) $\infty^4$ is $N$.  For any $a$, $(\infty, \infty, 1, 1, \infty, a)$ is $N$ because 
 
 \textbf{Claim:} $(a, \infty, 1, 1, \infty, a)$ is $P$.  
 
 To prove the claim, note that when $a=0$, it is implied by Theorem~\ref{outerinfinities}.  For positive $a$, there are options of the form $(a, \infty, 1, 1, \infty, a-k)$, which are $N$ by induction because they have the $P$ option $(a-k, \infty, 1, 1, \infty, a-k)$.  The only other option is $(a,\infty, \infty, a)$, which has a $P$ option $(a,b,b,a)$ by Theorem~\ref{4symmetric}.   Thus, the claim is proved.
 
 The remaining options involve decreasing two infinite stacks.  Reducing them to $0$, $1$, or $2$ are $N$ positions because they have the options $(2,2,1,1)$, $(1,1,1,1,1,1)$, and $(1,1,2,2)$ respectively, all of which are $P$ by Theorem~\ref{general}.  For $a>2$, the position $(\infty, \infty, 1, 1, a, a)$ has a $P$ option $(b, \infty, 1, 1, a, a)$ for $b$ which is the misere foreclosed Grundy value for the modified game on $(1,1,a,a)$, as discussed after Lemma~\ref{lastinfinity}.  
 \end{proof}
 
 We conclude with the natural next step. 
 
 \begin{openp}
 Is $(\infty,\infty,\infty,\infty,\infty,\infty, \infty)$ $P$ or $N$?
 \end{openp}

\bibliographystyle{plain}
\bibliography{wytbib}

\end{document}